\documentclass[12pt,twoside]{amsart}

\usepackage{amsmath, amssymb, amscd, paralist,tabularx,supertabular,amsmath, verbatim, amsthm, amssymb, mathrsfs, manfnt,times,latexsym,amscd,graphicx, color}
\usepackage[all]{xy}

\usepackage{fullpage}

\DeclareMathOperator{\M}{M}

\DeclareMathOperator{\PGL}{PGL}

\DeclareMathOperator{\Ram}{Ram}

\DeclareMathOperator{\Reg}{Reg}

\DeclareMathOperator{\vol}{vol}

\newcommand{\frakp}{\mathfrak{p}}

\newcommand{\frakD}{\mathfrak{D}}

\numberwithin{equation}{section}

\theoremstyle{remark}

\theoremstyle{plain}

\newtheorem{remark}[equation]{Remark}
\newtheorem{theorem}[equation]{Theorem}
\newtheorem*{thmnone}{Theorem}

\newtheorem{lem}[equation]{Lemma}
\newtheorem{cor}[equation]{Corollary}

\title{Families of mutually isospectral Riemannian orbifolds}

\author{Benjamin Linowitz}\thanks{Author was supported by NSF RTG grant DMS-1045119 and an NSF Mathematical Sciences Postdoctoral Fellowship.}
\address{Department of Mathematics\\
530 Church Street\\
University of Michigan\\
Ann Arbor, MI 48109 USA}
\email[] {linowitz@umich.edu}

\begin{document}

	\subjclass[2010] {Primary 58J53; Secondary 11F06}

	\keywords{isospectral orbifolds, arithmetic subgroups, Vign{\'e}ras' method}

\begin{abstract} 
In this paper we consider three arithmetic families of isospectral non-isometric Riemannian orbifolds and in each case derive an upper bound for the size of the family which is polynomial as a function of the volume of the orbifolds. The first family that we consider are those constructed by Vign{\'e}ras' method. The second and third families are those whose covering groups are the minimal covolume arithmetic subgroups and maximal arithmetic subgroups of $\PGL_2(\mathbb R)^a\times \PGL_2(\mathbb C)^b$.
\end{abstract}

\maketitle

\section{Introduction}

Let $M$ be a complete, orientable Riemannian orbifold of finite volume. The spectrum of $M$ is the set of eigenvalues of the Laplace-Beltrami operator acting on $L_2(M)$. It is known that this spectrum is discrete and that every eigenvalue occurs with a finite multiplicity. Two orbifolds are said to be \textit{isospectral} if their spectra coincide. A natural problem is to determine the extent to which the spectrum of $M$ determines its geometry and topology. It is well known for instance, that both volume and scalar curvature are spectral invariants. Isometry class, on the other hand, is not a spectral invariant. This was first shown by Milnor \cite{milnor}, who exhibited sixteen dimensional flat tori which were isospectral and non-isometric. For an excellent survey of the history of constructing isospectral non-isometric Riemannian manifolds and orbifolds we refer the reader to \cite{gordon-survey}.

At present there are only two systematic methods known for constructing isospectral non-isometric manifolds. The first is due to Vign{\'e}ras \cite{vigneras-isospectral} and was used to exhibit isospectral non-isometric hyperbolic $2$- and $3$-manifolds. Vign{\'e}ras' method is based upon the arithmetic of orders in quaternion algebras defined over algebraic number fields. The second method, due to Sunada \cite{sunada}, is extremely versatile and, along with its variants and generalizations, accounts for the majority of the known examples of isospectral non-isometric manifolds. Using Sunada's method, Brooks, Gornet and Gustafson \cite{brooks} have constructed arbitrarily large families of isospectral non-isometric Riemann surfaces. In fact, they prove the existence of a constant $c>0$ and an infinite sequence $g_i\rightarrow \infty$ such that for each $i$ there exist at least $g_i^{c\log(g_i)}$ mutually isospectral non-isometric Riemann surfaces of genus $g_i$. We note that similar results have recently been obtained by McReynolds \cite{mcreynolds} in the context of complex hyperbolic $2$-space and real hyperbolic $n$-space.

It is a result of McKean \cite{mckean} that for a fixed Riemann surface $S$, there are at most finitely many Riemann surfaces which are isospectral to $S$ and pairwise non-isometric. A similar result was proven by Kim \cite{kim} for hyperbolic $3$-manifolds. In his book Buser \cite{buser} provides an explicit bound for the number of surfaces which are isospectral to a fixed surface $S$ of genus $g$ and pairwise non-isometric; there are at most $\exp(720g^2)$ such surfaces.

In this paper we consider various natural families of Riemannian orbifolds and provide upper bounds (in terms of volume) for the cardinality of a set of pairwise isospectral non-isometric orbifolds whose elements lie in these families. All of the families we consider are arithmetic, and it is significant that in each case the upper bound we deduce is polynomial as a function of volume.

Our first result considers families of isospectral non-isometric orbifolds constructed via Vign{\'e}ras' method:

\begin{thmnone}
Fix $\epsilon>0$. The cardinality of a family of pairwise isospectral non-isometric Riemannian orbifolds of volume $V$ constructed via Vign{\'e}ras' method is at most $V^{2+\epsilon}$ for all $V\gg 0$.
\end{thmnone}

In Sections \ref{section:minimalvolume} and \ref{section:maximallattices} we consider families of maximal arithmetic subgroups of $\PGL_2(\mathbb R)^a\times \PGL_2(\mathbb C)^b$. These groups comprise an extremely important class of arithmetic subgroups. It is known, for instance, that the hyperbolic $3$-orbifold and $3$-manifold of smallest volume are both arithmetic \cite{chinburg-smallestorbifold, chinburg-smallestmanifold}. In Section \ref{section:minimalvolume} we consider maximal arithmetic subgroups which have minimal covolume in their commensurability class and prove:

\begin{thmnone}
Fix a commensurability class of cocompact arithmetic lattices in $\PGL_2(\mathbb R)^a\times \PGL_2(\mathbb C)^b$. The cardinality of a family of pairwise isospectral non-isometric arithmetic Riemannian orbifold quotients of $\mathcal H=\mathbb H_2^a\times \mathbb H_3^b$ which have minimal volume $V$ in this commensurability class is at most $242V^{18}$ for all $V>0$.
\end{thmnone}

In Section \ref{section:maximallattices} we consider the family of all maximal arithmetic subgroups and prove a bound which is again polynomial in volume. In fact, our result is somewhat stronger in the sense that we prove that within a fixed commensurability class, the set of all maximal arithmetic subgroups which have covolume less than $V$ is bounded by a function polynomial in $V$. Although the existence of such a bound is implicit in the proof of Theorem 1.6 of Belolipetsky, Gelander, Lubotzky and Shalev \cite{BGLS}, no explicit bound is currently known. We rectify this situation by proving:

\begin{thmnone}
Fix a commensurability class of cocompact arithmetic lattices in $\PGL_2(\mathbb R)^a\times \PGL_2(\mathbb C)^b$. The number of conjugacy classes of maximal arithmetic lattices with covolume at most $V$ in this commensurability class is less than $242V^{20}$ for all $V>0$.
\end{thmnone}

The proofs of our theorems make extensive use of Borel's volume formula \cite{borel-commensurability}, the work of Chinburg and Friedman \cite{chinburg-smallestorbifold}, a refinement of the Odlyzko discriminant bounds \cite{Odlyzko-bounds} due to Poitou \cite{Poitou}, and the Brauer-Seigel theorem. Along the way we prove a number theoretic result of independent interest which bounds the class number of an algebraic number field in terms of the absolute value of the discriminant of the field.

\section{Notation}

Throughout this paper $k$ will denote a number field of degree $n$ with signature $(r_1,r_2)$. That is, $k$ has $r_1$ real places and $r_2$ complex places. As a consequence $n=r_1+2r_2$. We will denote by $V_{\infty}$ the set of archimedean places of $k$. The class number of $k$ will be denoted by $h_k$ and the absolute value of the discriminant of $k$ will be denoted $d_k$. If $\mathfrak p$ is a prime of $k$ then we will denote by $N(\mathfrak p)$ the norm of $\mathfrak p$.

For a quaternion algebra $B$ over $k$ we define $\Ram(B)$ to be the set of primes of $k$ (possibly infinite) which ramify in $B$, $\Ram_f(B)$ to be the set of finite primes of $k$ lying in $\Ram(B)$ and $\Ram_{\infty}(B)$ to be the set of archimedean places of $k$ lying in $\Ram(B)$. The discriminant of $B$, which we will denote by $\frakD$, is defined to be the product of the primes lying in $\Ram_f(B)$.

\section{Isospectral orbifolds obtained via Vign{\'e}ras' method}

We begin by describing Vign{\'e}ras' construction \cite{vigneras-isospectral} of isospectral non-isometric Riemannian orbifolds. For a more detailed exposition, see Chapter 12 of \cite{mac-reid-book}.

Let $k$ be a degree $n$ number field of signature $(r_1,r_2)$ and let $B$ be a quaternion algebra over $k$ which is not ramified at all archimedean places of $k$. There exists an isomorphism $$B\otimes_{\mathbb Q} \mathbb R \cong \mathbb H^r \times \M_2(\mathbb R)^s \times \M_2(\mathbb C)^{r_2}, \qquad r+s=r_1,$$ which induces an embedding $$B^\times \hookrightarrow \prod_{\nu\in V_{\infty}\setminus \Ram_{\infty}(B)} B_{\nu}^\times.$$ This embedding in turn induces an embedding $\rho: B^\times/k^\times \hookrightarrow G=\PGL_2(\mathbb R)^{s}\times\PGL_2(\mathbb C)^{r_2}$. Let $K$ be a maximal compact subgroup of $G$ so that $G/K=\mathbb H_2^s \times \mathbb H_3^{r_2}$ is a product of two and three dimensional hyperbolic spaces. The Riemannian orbifolds constructed by Vign{\'e}ras are of the form $\Gamma\backslash G/K$ for $\Gamma$ a discrete subgroup of isometries of $G/K$.

Let $\mathcal O_k$ be the ring of integers of $k$ and $\mathcal O$ be a maximal $\mathcal O_k$-order of $B$. Denote by $\mathcal O^1$ the multiplicative group of elements of $\mathcal O$ with reduced norm one. Then $\Gamma_{\mathcal O}^1=\rho(\mathcal O^1)$ is a discrete subgroup of isometries of $G/K$ of finite covolume which is cocompact if $B$ is a division algebra.

Suppose now that $B$ is as above and that $\mathcal O$ and $\mathcal O^\prime$ are maximal orders of $B$. It is clear that if there exists $\alpha\in B^\times$ such that $\mathcal O^\prime=\alpha\mathcal O\alpha^{-1}$ then $\Gamma_{\mathcal O}^1$ and $\Gamma_{\mathcal O^\prime}^1$ will be conjugate and consequently that $\Gamma_{\mathcal O}^1\backslash G/K$ and $\Gamma_{\mathcal O^\prime}^1\backslash G/K$ will be isometric. What is not obvious and is the content of \cite[Th{\'e}or{\`e}me 3]{vigneras-isospectral} is that if $\Gamma_{\mathcal O}^1\backslash G/K$ and $\Gamma_{\mathcal O^\prime}^1\backslash G/K$ are isometric then $\mathcal O$ and $\mathcal O^\prime$ are conjugate by an element of $B^\times$. Furthermore, Vign{\'e}ras \cite[Th{\'e}or{\`e}me 7]{vigneras-isospectral} showed that when a certain technical number theoretic condition is satisfied, the spectrum of $\Gamma_{\mathcal O}^1\backslash G/K$ is independent of the choice of maximal order $\mathcal O$. Although we will not state this number theoretic condition explicitly because of the burdensome notation that it would require, we do remark that if $\Ram_f(B)$ is nonempty then this condition will be satisfied.

Before stating and proving this section's main theorem we prove a lemma which will be used throughout this paper and which is of independent interest.

\begin{lem}\label{lem:BP}
Let $k$ be a number field of signature $(r_1,r_2)$ with class number $h_k$ and absolute value of discriminant $d_k$. Then $h_k\leq 242d_k^{3/4}/ (1.64)^{r_1}$.
\end{lem}
\begin{proof}
The analytic class number formula and the Brauer-Siegel theorem \cite[pp. 300, 322]{Lang-ANT} imply that for any real $s>0$ we have $$h_k \leq \frac{\omega_k s (s-1)\Gamma(s)^{r_2}\Gamma(\frac{s}{2})^{r_1}\zeta_k(s)d_k^{\frac{s}{2}}}{2^{r_1}\Reg_k2^{r_2s}\pi^{\frac{ns}{2}}},$$ where $\omega_k$ denotes the number of roots of unity contained in $k$, $\zeta_k(s)$ is the Dedekind zeta function of $k$ and $\Reg_k$ is the regulator of $k$.

Using the well-known estimate $\zeta_k(s)<\zeta(s)^n$ we obtain, for $s=1.5$, the inequality $$h_k < \frac{3 \omega_k \left(2.62\right)^n d_k^{\frac{3}{4}}}{4\Reg_k2^{\frac{3r_2}{2}}\pi^{\frac{3n}{4}}}.$$

Friedman \cite[pp. 620]{Friedman} has shown that $\Reg_k\geq 0.0031\omega_k\exp(0.241n+0.497r_1)$. Combining this and the fact that $\exp(0.497)>1.64$ with the estimate above and simplifying gives 

$$h_k < \frac{242 d_k^{\frac{3}{4}}}{(1.64)^{r_1} 2^{\frac{3r_2}{2}}},$$

from which the proposition easily follows. \end{proof}

\begin{remark}
Lemma \ref{lem:BP} generalizes a result of Borel and Prasad \cite[pp. 143]{BP} which shows that $h_k\leq 10^2 \left(\frac{\pi}{12}\right)^n d_k$.
\end{remark}

\begin{theorem}\label{theorem:vignerastheorem}
Fix $\epsilon>0$. The cardinality of a family of pairwise isospectral non-isometric Riemannian orbifolds of volume $V$ constructed via Vign{\'e}ras' method is at most $V^{2+\epsilon}$ for all sufficiently large $V$.
\end{theorem}
\begin{proof}
All of the orbifolds obtained via Vign{\'e}ras' method have covering groups which are of the form $\Gamma^1_{\mathcal O}$, for $\mathcal O$ a maximal order in a quaternion algebra $B$ defined over a number field $k$. If $\mathcal O^\prime$ is another maximal order of $B$, then $\Gamma^1_{\mathcal O^\prime}$ will be conjugate to $\Gamma^1_{\mathcal O}$ precisely when $\mathcal O$ and $\mathcal O^\prime$ are conjugate in $B$. It therefore follows that the number of elements in a family of pairwise isospectral non-isometric orbifolds constructed via Vign{\'e}ras' method is bounded above by the number of conjugacy classes of maximal orders in $B$. It is well known that this latter number (often referred to as the \textit{type number} of $B$) coincides with the degree of the maximal abelian extension of $k$ which has $2$-elementary Galois group, is unramified outside of the real places in $\Ram(B)$ and in which all finite primes of $\Ram(B)$ split completely \cite[pp. 37, 39]{Chinburg-Friedman}. This latter field extension is contained in the strict class field of $k$, which has degree at most $2^{r_1} h_k$.

We now relate $2^{r_1} h_k$ to the discriminant $d_k$ of $k$, which we will in turn relate to the volume $V$. Lemma \ref{lem:BP} shows that $h_k\leq 242 (1.64)^{-r_1} d_k^{\frac{3}{4}}$, hence $2^{r_1} h_k \leq 242 (1.22)^{r_1} d_k^{\frac{3}{4}}$.

We now recall the formula for the co-volume of $\Gamma^1_{\mathcal O}$ ( see Borel \cite[7.3]{borel-commensurability}, Maclachlan and Reid \cite[\S 11.1]{mac-reid-book}, Vign\'eras \cite[Corollaire IV.1.8]{vigneras-book} and, for the totally real case, Shimizu \cite[Appendix]{Shimizu}  ):

\begin{equation}\label{equation:volumeformula}
\vol(\mathcal H/\Gamma_{\mathcal O}^1)=\frac{2(4\pi)^s d_k^{3/2} \zeta_k(2) \Phi(\frakD)}{(4\pi^2)^{r_1}(8\pi^2)^{r_2}} ,
\end{equation}
where $\Phi(\frakD) = N(\frakD) \prod_{\frakp \mid \frakD} \left(1-\frac{1}{N(\frakp)}\right)$.

Applying the trivial estimates $\zeta_k(2), \Phi(\frakD), (4\pi)^s \geq 1$ we see that 
\begin{equation}\label{equation:firstvolume}
V\geq \frac{d_k^{3/2}}{(4\pi^2)^{r_1}(8\pi^2)^{r_2}}.
\end{equation}
 In order to get rid of the dependence on $r_1$ and $r_2$ in (\ref{equation:firstvolume}) we apply the discriminant bounds of Odlyzko \cite{Odlyzko-bounds} and Poitou \cite{Poitou} (see also \cite[Section 2]{doud}). In particular there exists an absolute constant $C$ such that $\log(d_k)\geq r_1+n(\gamma+\log(4\pi))-C$, where $\gamma=0.57721\dots$ is the Euler-Mascheroni constant. As $n=r_1+2r_2$ we deduce that $\log(d_k)\geq 4r_1+6r_2-C$, and because $4> \log(4\pi^2)$ and $6>\log(8\pi^2)$, we conclude from (\ref{equation:firstvolume}) that $VC_1\geq d_k^{1/2}$ for some absolute constant $C_1$. 

We have already shown that the cardinality of a family of pairwise isospectral non-isometric Riemannian orbifolds of volume $V$ constructed via Vign{\'e}ras' method is at most $242 (1.22)^{r_1} d_k^{\frac{3}{4}}$. We have also seen that $\log(d_k)\geq 4r_1-C$ where $C$ is an absolute constant. Easy computations based on the formulas in Section 2 of \cite{doud} show that one may take $C=4.5$. It follows that $d_k\geq \exp(4r_1)/\exp(4.5)$. To ease notation slightly, set $f(r_1)=\exp(4r_1)/\exp(4.5)$. Then it is easy to see that $(1.22)^{r_1} \leq f(r_1)^{\frac{1}{4}}$ for all $r_1>1$. It follows that $242(1.22)^{r_1}\leq 242f(r_1)^{\frac{1}{4}}\leq 242 d_k^{\frac{1}{4}}$, hence the cardinality of a family of pairwise isospectral non-isometric Riemannian orbifolds of volume $V$ constructed via Vign{\'e}ras' method is at most $242d_k\leq 242C_1^2V^2$. The theorem follows.\end{proof}

\section{Isospectral minimal volume orbifolds}\label{section:minimalvolume}

We begin this section by setting up the notation needed to define the maximal arithmetic lattices in the commensurability class $\mathcal C(k,B)$ defined by a quaternion algebra $B$ over $k$. Our description will necessarily be brief. For more details we refer the reader to Borel \cite{borel-commensurability}, Chinburg and Friedman \cite[pp. 41]{Chinburg-Friedman} and Maclachlan and Reid \cite[Section 11]{mac-reid-book}.

Let $S$ be a finite set of primes of $k$ which is disjoint from $\Ram_f(B)$. For each prime $\mathfrak p\in S$ let $\{ M^1_{\mathfrak p},M^2_{\mathfrak p} \}$ represent an edge in the tree of maximal orders of $\M_2(k_\mathfrak p)$. Given this notation we define $\Gamma_{S,\mathcal O}$ to be the image in $\PGL_2(\mathbb R)^a\times \PGL_2(\mathbb C)^b$ of 
$$\{ \overline{x}\in B^\times/k^\times : x\mathcal O_{\mathfrak p}x^{-1}=\mathcal O_{\mathfrak p} \mbox{ for }\mathfrak p\not\in S \mbox{ and } x \mbox{ fixes } \{ M^1_{\mathfrak p},M^2_{\mathfrak p} \} \mbox{ for } \mathfrak p \in S  \}.$$

When $S=\emptyset$ we define $\Gamma_{\mathcal O}=\Gamma_{\emptyset,\mathcal O}$. These groups were first defined by Borel \cite{borel-commensurability}, who showed that while all of these groups are not maximal arithmetic subgroups of $\mathcal C(k,B)$, every maximal arithmetic subgroup of $\mathcal C(k,B)$ is conjugate to some $\Gamma_{S,\mathcal O}$. Borel also derived formulas for the covolumes of these groups, showing in particular that the group $\Gamma_{\mathcal O}$ has minimal covolume in its commensurability class.

\begin{theorem}\label{theorem:minimalvolume}
Fix a commensurability class of cocompact arithmetic lattices in $\PGL_2(\mathbb R)^a\times \PGL_2(\mathbb C)^b$. The cardinality of a family of pairwise isospectral non-isometric arithmetic Riemannian orbifold quotients of $\mathcal H=\mathbb H_2^a\times \mathbb H_3^b$ which have minimal volume $V$ in this commensurability class is at most $242V^{18}$ for all $V>0$.
\end{theorem}

\begin{proof}
If $M$ is an arithmetic orbifold quotient of $\mathcal H$ which has minimal volume in its commensurability class then there exists a number field $k$, quaternion algebra $B$ over $k$ and maximal order $\mathcal O$ of $B$ such that the covering group of $M$ is conjugate to $\Gamma_{\mathcal O}$. If $\mathcal O^\prime$ is another maximal order of $B$ then $\Gamma_{\mathcal O^\prime}$ is conjugate to $\Gamma_{\mathcal O}$ if and only if $\mathcal O$ and $\mathcal O^\prime$ are conjugate in $B$. As in the proof of Theorem \ref{theorem:vignerastheorem} it follows that the maximum number of orbifolds which are all isospectral to $M$ and are mutually non-isometric is bounded above by $242(1.22)^{r_1} d_k^{\frac{3}{4}}$.

We now derive a lower bound for the volume $\vol(\mathcal H/\Gamma_{\mathcal O})$ of $M$. In order to do so we note that by the results of Borel \cite[Sections 8.4 and 8.5]{borel-commensurability} and Chinburg and Friedman \cite[Lemma 2.1]{chinburg-smallestorbifold} we have that $[\Gamma_{\mathcal O}:\Gamma_{\mathcal O}^1]\leq 2^{n+|\Ram_f(B)|}h_k$. From this we deduce that

\begin{equation}\label{equation:minimalvolume}
	V=\vol(\mathcal H/\Gamma_{\mathcal O})\geq \frac{2(4\pi)^s d_k^{3/2} \zeta_k(2) \Phi(\frakD)}{(4\pi^2)^{r_1}(8\pi^2)^{r_2}2^{n+|\Ram_f(B)|}h_k}
\end{equation}

Let $\omega_2(B)$ denote the number of primes of $k$ which have norm $2$ and ramify in $B$. Then $\Phi(\frakD)/2^{|\Ram_f(B)|}\geq (\frac{1}{2})^{\omega_2(B)}$ and $\zeta_k(2)\geq (4/3)^{\omega_2(B)}$ (the latter is an immediate consequence of the Euler product expansion of $\zeta_k(s)$). Using this, Lemma \ref{lem:BP}, the trivial bound $\omega_2(B)\leq n$ and simplifying, we see that

\begin{align}  V & \geq \frac{d_k^{3/4}}{25^{r_1}(8\pi^2)^{r_2}3^{n}} \\ & \geq \frac{d_k^{3/4}}{(75)^{n}} \end{align}

An easy consequence of Lemma 4.3 of \cite{chinburg-smallestorbifold} is that $3\log(V)\geq n$. Equation (4.4) now implies that $d_k\leq V^{22}$. As $1.22<\exp(\frac{1}{4})$, we see that $242(1.22)^{r_1}\leq 242\exp(\frac{n}{4})\leq 242\exp(\frac{\log(V^3)}{4})=242V^{\frac{3}{4}}$. Putting all of this together we see that $242(1.22)^{r_1}d_k^{\frac{3}{4}}\leq 242V^{\frac{3}{4}}V^{\frac{33}{2}}\leq 242 V^{18}$, finishing the proof.\end{proof}

\section{Isospectral maximal arithmetic lattices}\label{section:maximallattices}

In Section \ref{section:minimalvolume} we derived an upper bound for the cardinality of a family of pairwise isospectral non-isometric arithmetic Riemannian orbifolds whose members all have covering groups of the form $\Gamma_{\mathcal O}$. In this section we consider Riemannian orbifolds whose covering groups belong to the broader class of maximal arithmetic lattices in $\PGL_2(\mathbb R)^a\times \PGL_2(\mathbb C)^b$.

\begin{theorem}\label{theorem:maximaltheorem}
Fix a commensurability class $\mathcal C(k,B)$ of cocompact arithmetic lattices in $\PGL_2(\mathbb R)^a\times \PGL_2(\mathbb C)^b$. The number of conjugacy classes of maximal arithmetic lattices with covolume at most $V$ in this commensurability class is less than $242V^{20}$ for all $V>0$.
\end{theorem}

\begin{proof}
Let $S$ be a finite set of primes of $k$ which are disjoint from $\Ram_f(B)$ and let $\mathcal O$ be a maximal order of $B$. Then $\Gamma_{S,\mathcal O}$ is an arithmetic lattice in $\mathcal C(k,B)$ and, while $\Gamma_{S,\mathcal O}$ may not be maximal, every maximal arithmetic lattices in $\mathcal C(k,B)$ is of this form. If $\mathcal O^\prime$ is another maximal order of $B$ then $\Gamma_{S,\mathcal O}$ is conjugate to $\Gamma_{S,\mathcal O^\prime}$ if and only if $\mathcal O$ and $\mathcal O^\prime$ are conjugate in $B$. It follows that the number of conjugacy classes of maximal arithmetic lattices of volume at most $V$ in $\mathcal C(k,B)$ is less than the type number of $B$ times the number of choices for the set $S$. We have already seen that the type number of $B$ is less than $242(1.22)^{r_1}d_k^{\frac{3}{4}}.$

Before proceeding we require a definition. Let $\Gamma_1,\Gamma_2\in \mathcal C(k,B)$. We define the \textit{generalized index} $[\Gamma_1:\Gamma_2]\in\mathbb Q$ to be the quotient $\frac{[\Gamma_1:\Gamma_1\cap\Gamma_2]}{[\Gamma_2:\Gamma_1\cap\Gamma_2]}$.

Maclachlan and Reid \cite[Theorem 11.5.1]{mac-reid-book} have shown that there exists an integer $0\leq m \leq |S| $ such that $[\Gamma_{\mathcal O}:\Gamma_{S,\mathcal O}]=2^{-m}\prod_{\mathfrak p\in S}\left(N(\mathfrak p)+1\right)$. Set $\mathcal H=\mathbb H_2^a\times \mathbb H_3^b$ so that 

\begin{align*}
\frac{V}{\vol(\mathcal H/\Gamma_{\mathcal O})} & \geq \frac{\vol(\mathcal H/\Gamma_{S,\mathcal O})}{\vol(\mathcal H/\Gamma_{\mathcal O})} \\
& \geq \prod_{\mathfrak p\in S} \frac{\left(N(\mathfrak p)+1\right)}{2} \\
& \geq \prod_{\substack{\mathfrak p\in S \\ N(\mathfrak p)\neq 2}} N(\mathfrak p)^{\frac{1}{3}}.
\end{align*}

In the proof of Theorem \ref{theorem:minimalvolume} we proved that $\vol(\mathcal H/\Gamma_{\mathcal O})\geq d_k^{\frac{1}{22}}$, hence the above inequalities show that \begin{equation}\label{equation:Sbound}
\frac{V^3}{d_k^{\frac{3}{22}}}\geq \prod_{\substack{\mathfrak p\in S \\ N(\mathfrak p)\neq 2}} N(\mathfrak p).
\end{equation}
Equation (\ref{equation:Sbound}) shows that the number of choices for the set $S$ is less than the number of integral ideals of norm less than $\frac{V^3}{d_k^{\frac{3}{22}}}$, which by \cite[Lemma 3.4]{BGLS} is less than $\left(\frac{\pi^2}{6}\right)^n \frac{V^6}{d_k^{\frac{3}{11}}}$. Therefore the number of conjugacy classes of maximal arithmetic lattices of volume at most $V$ in $\mathcal C(k,B)$ is less than $242(1.22)^{r_1}\left(\frac{\pi^2}{6}\right)^n d_k^{\frac{21}{44}}V^6$. 

To ease notation let $V^\prime=\vol(\mathcal H/\Gamma_{\mathcal O})$. Using the fact that $3\log(V^\prime)\geq n$ (see \cite[Lemma 4.3]{chinburg-smallestorbifold}) we deduce that the number of conjugacy classes of maximal arithmetic lattices of volume at most $V$ in $\mathcal C(k,B)$ is less than $242{V^\prime}^3V^6d_k^{\frac{21}{44}}$. As we have seen that $d_k^{\frac{1}{22}}\leq V^\prime$ and trivially have $V^{\prime}\leq V$, we deduce that $d_k^{\frac{21}{44}}\leq {V^\prime}^{\frac{21}{2}}$, hence $242{V^\prime}^3V^6d_k^{\frac{21}{44}}\leq 242{V^\prime}^{\frac{27}{2}}V^6\leq 242 V^{20}$. \end{proof}

Because isospectral Riemannian orbifolds must have the same volume, the following is an immediate consequence of Theorem \ref{theorem:maximaltheorem}.

\begin{cor}\label{cor:isospectralcor}
The cardinality of a family of commensurable, pairwise isospectral non-isometric arithmetic Riemannian orbifold quotients of $\mathcal H=\mathbb H_2^a\times \mathbb H_3^b$ with covering groups which are maximal in their commensurability class and have volume $V$ is at most $242 V^{20}$ for all $V>0$.
\end{cor}

\begin{remark}
Reid \cite{Reid} has shown that isospectral arithmetic hyperbolic $2$- and $3$-orbifolds are always commensurable, hence when $a+b=1$ the commensurability assumption in Theorem \ref{theorem:minimalvolume} and Corollary \ref{cor:isospectralcor} can be omitted.
\end{remark}

\end{document}